\DeclareSymbolFont{AMSb}{U}{msb}{m}{n}
\documentclass[reqno,centertags,11pt,a4paper,noamsfonts]{amsart}
\pdfoutput=1
\usepackage[svgnames]{xcolor}
\usepackage{graphicx}
\usepackage[english]{babel}
\usepackage[babel=true,expansion=alltext,protrusion=alltext-nott,final]{microtype}
\usepackage[margin={3cm},marginparwidth={2.5cm}]{geometry}
\usepackage[utf8]{inputenc}
\usepackage{libertine}
\usepackage[libertine]{newtxmath}
\useosf
\usepackage[varqu,varl]{inconsolata}
\usepackage[T1]{fontenc}
\usepackage{textcomp}
\usepackage{amsthm}
\usepackage[inline]{enumitem}
\numberwithin{equation}{section}
\usepackage[normalem]{ulem}
\usepackage{tikz}
\usetikzlibrary{arrows}
\usepackage{bm}
\usepackage{soul}

\providecommand{\mr}[1]{\href{http://www.ams.org/mathscinet-getitem?mr=#1}{MR~#1}}
\providecommand{\zbl}[1]{\href{https://zbmath.org/?q=an:#1}{Zbl~#1}}

\providecommand{\doi}[2]{\href{https://doi.org/#1}{#2}}
\newcommand{\RR}{\mathbb{R}}
\newcommand{\RRd}{\mathbb{R}^d}

\newcommand{\NN}{\mathbb{N}}

\newcommand{\HH}{\mathcal{H}}

\newcommand{\LL}{\mathcal{L}}

\newcommand{\1}{\mathbf{1}}

\newcommand{\ytilde}{\tilde{y}}

\newcommand{\xbar}{\overline{x}}

\newcommand{\subscript}[2]{$#1 _ #2$}

\newcommand{\eps}{\varepsilon}
\newcommand{\xdots}{x_{1},\dots,x_{N}}
\newcommand{\ydots}{y_{1},\dots,y_{N}}
\newcommand{\zdots}{z_{1},\dots,z_{N}}
\newcommand{\ytildedots}{\tilde{y}_{1},\dots,\tilde{y}_{N}}
\newcommand{\Prob}{\mathcal{P}}
\newcommand{\xh}{\overline{x}_h}

\newcommand{\xvec}{\mathbf{x}}
\newcommand{\yvec}{\mathbf{y}}
\newcommand{\ytildevec}{\tilde{\mathbf{y}}}
\newcommand{\zvec}{\mathbf{z}}
\DeclareMathOperator{\weaklystar}{\rightharpoonup\kern-2.2ex ^* \, \,}

\DeclareMathOperator{\argmin}{argmin}

\DeclareMathOperator{\spt}{spt}
\DeclareMathOperator{\diam}{diam}

\def\XXint#1#2#3{{\setbox0=\hbox{$#1{#2#3}{\int}$ }
\vcenter{\hbox{$#2#3$ }}\kern-.6\wd0}}

\definecolor{darkred}{rgb}{0.4,0,0} 
\definecolor{darkgreen}{rgb}{0,.4,0}

\theoremstyle{plain}
\newtheorem{theorem}{Theorem}[section]
\newtheorem{proposition}[theorem]{Proposition}
\newtheorem{corollary}[theorem]{Corollary}
\newtheorem{lemma}[theorem]{Lemma}

\newtheorem*{theorem*}{Theorem}

\theoremstyle{definition}
\newtheorem{definition}[theorem]{Definition}
\newtheorem{remark}[theorem]{Remark}

\usepackage[pagebackref=false, debug=false]{hyperref}
\hypersetup{
	linktoc=all,
	colorlinks=true,
	linkcolor=Green,citecolor=Orange,urlcolor=DarkBlue,
	pdftitle={h-Wasserstein barycenters},
	pdfauthor={Camilla Brizzi, Gero Friesecke, Tobias Ried}, 
	pdfsubject={Wasserstein barycenter, optimal transport,  multi-marginal optimal transport},
	pdfkeywords={Wasserstein barycenter, optimal transport,  multi-marginal optimal transport}} 
\begin{document}
\title{$h$-Wasserstein barycenters}

\author[C.~Brizzi]{Camilla Brizzi}
\address[C.~Brizzi]{Technische Universität München, Department of Mathematics, Boltzmannstraße 3, 85748 Garching, Germany}
\email{briz@ma.tum.de}

\author[G.~Friesecke]{Gero Friesecke}
\address[G.~Friesecke]{Technische Universität München, Department of Mathematics, Boltzmannstraße 3, 85748 Garching, Germany}
\email{friesecke@tum.de}

\author[T.~Ried]{Tobias Ried}
\address[T.~Ried]{Max-Planck-Institut für Mathematik in den Naturwissenschaften, Inselstraße 22, 04103 Leipzig, Germany}
\email{Tobias.Ried@mis.mpg.de}
\keywords{Wasserstein barycenter, optimal transport,  multi-marginal optimal transport}
\subjclass[2020]{Primary 49Q20; Secondary 49J40, 49K21} 
\date{\today}
\thanks{\emph{Funding information}: Deutsche Forschungsgemeinschaft (DFG -- German Research Foundation) -- Project-ID 195170736 -- TRR109.\\[1ex]
\textcopyright 2024 by the authors. Faithful reproduction of this article, in its entirety, by any means is permitted for noncommercial purposes.}
\begin{abstract}
	We generalize the notion and theory of Wasserstein barycenters \cite{AC11} from the quadratic cost to general smooth strictly convex costs $h$ with non-degenerate Hessian. 
    We show the equivalence between a coupled two-marginal and a multi-marginal formulation and establish that the multi-marginal optimal plan is unique and of Monge form. 
    To establish the latter result we introduce a new approach which is not based on explicitly solving the optimality system, but instead deriving a quantitative injectivity estimate for the (highly non-injective) map from $N$-point configurations to their $h$-barycenter on the support of an optimal multi-marginal plan.
\end{abstract}
\maketitle
{
\tableofcontents}
\section{Introduction}

Barycenters with respect to the 2-Wasserstein distance, introduced by Agueh and Carlier in \cite{AC11}, are an important concept in data science, statistics, and image processing, and also constitute an important example of multi-marginal optimal transport. Given $N$ input datasets, modelled by probability measures on $\RR^d$, Wasserstein barycenters provide
\begin{itemize}
    \item \emph{a representative summary} (by equal-weight barycentring, corresponding in dimension $d=1$ to histogram equalization);
	\item \emph{a continuous family of interpolates} (by varying-weight barycentring, corresponding for $N=2$ marginals to Wasserstein geodesics).
\end{itemize}

Wasserstein barycenters can be defined in two different ways: in terms of a multi-marginal optimal transport (MMOT) problem -- and therefore a linear programming problem (in a high-dimensional space) --, and in terms of a non-linear variational problem (coupled two-marginal formulation). 

More precisely, let $\mu_1,\dots,\mu_N$, $N\ge2$, be probability measures on $\RR^d$ and consider the multi-marginal optimal transport problem 
\begin{equation}\label{eq:W2-MMOT}
     \min_{\gamma\in\Pi(\mu_1,\dots,\mu_N)}\int_{\RR^{Nd}} \sum_{i=1}^N \lambda_i |x_i - \xbar_2(x_1, \dots, x_N)|^2 \,d\gamma(x_1,\dots,x_N),
 \end{equation}
where
 \begin{equation*}
     \Pi(\mu_1,\dots,\mu_N):=\{\gamma\in\Prob(\RR^{Nd}) \, : \, \pi^i_\sharp \gamma = \mu_i \}
 \end{equation*}
 is the set of admissible transport plans between the marginals $\mu_1, \dots, \mu_N$, and 
 \begin{equation} \label{eq:W2-costfctn}
    \xbar_2(x_1, \dots, x_N) := \argmin_{z\in\RR^d} \sum_{i=1}^N \lambda_i |x_i-z|^2(= \sum_{i=1}^N \lambda_i x_i)
\end{equation}
is the classical (Euclidean) barycenter of the points $x_1, \dots, x_N \in \RR^d$ with weights $\lambda_1, \dots, \lambda_N>0$ such that $\sum_{i=1}^N \lambda_i = 1$.

The existence of an optimal transport plan $\overline{\gamma}$ is a standard result in optimal transport which can be straightforwardly generalized to the multi-marginal setting, see for instance \cite{F24}. 
Given any optimal plan $\overline{\gamma}$, the probability measure $(\overline{x}_2)_{\#}\overline{\gamma}$ on $\RR^d$ is called $W_2$-barycenter of $\mu_1, \dots, \mu_N$ with weights $\lambda_1, \dots, \lambda_N$. 

Equivalently, one can consider the variational problem of minimizing the mean-squared error
\begin{align}\label{eq:W2-C2M}
	\min_{\nu \in \mathcal{P}(\RR^d)} \sum_{i=1}^N \lambda_i W_2^2(\mu_i, \nu)
\end{align}
over probability measures $\nu \in \mathcal{P}(\RR^d)$. Then the two minimization problems \eqref{eq:W2-MMOT} and \eqref{eq:W2-C2M} are equivalent in the sense that their optimal values agree and $\overline{\nu}$ is a minimizer of the coupled two-marginal formulation \eqref{eq:W2-C2M} if and only if $\overline{\nu} = (\overline{x}_2)_{\#}\overline{\gamma}$ for a minimizer $\overline{\gamma}$ of \eqref{eq:W2-MMOT}, see \cite[Proposition 4.2]{AC11}. This equivalence tells us that by paying the price of increasing the dimension we can transform an highly non-linear problem in a linear one.

That is why a main interesting feature of the MMOT problem associated with $W_2$-barycenters is the \emph{sparsity} of the minimizer, as proved in the seminal paper by Gangbo and \'Swie\k{}ch \cite{GS98}. If the marginals $\mu_1, \dots, \mu_N$ are absolutely continuous with respect to Lebesgue measure\footnote{or, as proved later by Agueh and Carlier \cite{AC11}, at least one of the marginals is absolutely continuous}, then the optimal plan, instead of being fully supported in $\RR^{dN}$, is supported on  a graph over the first marginal space, 
\begin{equation}\label{eq:GSform}
   \overline{\gamma} = (\mathrm{Id}, T_2, \dots, T_N)_{\#}\mu_1
\end{equation}
for some functions $T_i: \mathrm{supp} \mu_1 \to \mathrm{supp} \mu_i$, $i=2, \dots, N$. This amounts to a dramatic decrease of complexity. 

The Gangbo--\'Swie\k{}ch result can be viewed as a multi-marginal Brenier's theorem. The proof strongly relies on the quadratic form of the cost, allowing explicit manipulation of the optimality system and reducing complicated $c$-convexity properties of the Kantorovich potentials to standard convexity.

Here we generalize the Gangbo--\'Swie\k{}ch theorem to barycenters with respect to general convex interactions, corresponding to the following replacements in \eqref{eq:W2-MMOT} and \eqref{eq:W2-costfctn}:
\begin{eqnarray*}
   |x_i-\xbar|^2 & \rightsquigarrow & h(x_i-\xbar) \\
   \xbar = \argmin_z \sum_i \lambda_i |x_i-z|^2 & \rightsquigarrow &  \xbar = \argmin_z \sum_i \lambda_i h(x_i-z).
\end{eqnarray*}
We show that optimal plans are still of the form \eqref{eq:GSform}. This can be viewed as a multi-marginal version of the Gangbo-McCann theorem, see \cite{GMC96}. 

The proof is given 
\begin{enumerate}[label=(\roman*)]
	\item for strictly convex $\mathcal{C}^2$ functions $h$ with nondegenerate Hessian in this article,
	\item for the Wasserstein $W_p$ distances with $1<p<\infty$ in our companion paper \cite{BFR24}. 
\end{enumerate}

The main result in case (i) could alternatively be deduced from the more abstract results by Pass \cite{Pas14}, by checking that our cost \eqref{eq:W2-C2M}-\eqref{eq:W2-costfctn} satisfies the hypotheses in \cite{Pas14}. This approach relies on the deep result by Pass \cite{Pas12} on the dependence of the local support dimension of optimal plans on a matrix of certain (but not all) 2nd derivatives of the cost function. 
As in \cite{Pas12} we use  $c$-monotonicity, but avoid the machinery of geometric measure theory.
Our proof is new and we believe simpler, relying on a direct quantitative injectivity estimate of the map from $N$ points to their barycenter on the support of optimal plans, which in turn relies on a monotonicity property of this map with respect to each of the $N$ points. This injectivity has the interesting statistical meaning that the mass of the barycenter at any given point only comes from a unique $N$-point configuration carrying mass of the marginals. Our approach -- unlike that in \cite{Pas12} -- extends to degenerate/singular situations, in particular allowing the treatment of $W_p$-barycenters. An interesting insight in our analysis --  which we  actually use in our proof of \eqref{eq:GSform} -- is that the barycenter with respect to any convex cost as in (i) and (ii) is itself absolutely continuous. This generalizes  a corresponding insight of Agueh and Carlier \cite{AC11} for the quadratic cost.  

In a different direction, Kim and Pass \cite{KP17} 
prove the absolute continuity of classical Wasserstein barycenters on compact Riemannian manifolds by exploiting special properties of the Riemannian distance squared.\footnote{More generally, such a result holds on compact Alexandrov spaces with curvature at least $-1$, see \cite{Jia17}.}

Finally we hope that our new approach might also prove useful in other multi-marginal problems of interest such as the Coulomb cost \cite{FGG22} or the generalized Euler equations \cite{B89}, where the amount of sparsity of optimal plans is an open question.

\medskip
The rest of the article is structured as follows: In the next section we state our main result precisely. In Section~\ref{sec:preliminaries} we recall basic properties of multi-marginal optimal transport. The equivalence of the coupled two-marginal formulation and MMOT formulation of $h$-barycenters is presented in Section~\ref{sec:equivalence}. In Section~\ref{sec:main} we give the detailed proof of the absolute continuity of the $W_h$-barycenter, which is the main ingredient in the proof of Theorem~\ref{th:sparsityplan}.

\section{Main result}
\subsection{$h$-barycenters}
From now on, we consider functions $h:\RRd\to\RR$ with the following properties: 
\begin{enumerate}[label=(\subscript{A}{\arabic*})]
	\item\label{A1}\textit{nonnegativity}: $h\ge 0$,
	\item\label{A2}\textit {strict convexity}: $h$ strictly convex,
	\item\label{A3}\textit{coercivity}:  $\lim_{|z|\to+ \infty}h(z)=+\infty$.
\end{enumerate}
Under these assumptions, the notion of barycenter of given points $x_1, \dots, x_N \in \RR^d$, $N\geq 2$, with respect to weights $\lambda_1, \dots, \lambda_N >0$ such that $\sum_{i=1}^N \lambda_i = 1$, can easily be generalized.
\begin{definition}[$h$-barycenter]
We call the map $\xh:\RR^{Nd}\to\RRd$, defined by
\begin{equation}\label{eq:hbary} \tag{$h$-bar}
   \overline{x}_h(x_1,\dots,x_N):=\underset{z\in\RRd}\argmin\sum_{i=1}^{N}\lambda_i h(x_i-z),
\end{equation}
$h$-barycenter of the points $x_1, \dots, x_N$. 
\end{definition}
Observe that under the assumptions \ref{A1}-\ref{A3}, $\overline{x}_h(x_1,\dots,x_N)$ exists and is unique for every $\xvec=(\xdots)\in\RR^{Nd}$.

We now introduce the associated non-negative cost function $c_h:\RR^{Nd}\to\RR$ defined by  
\begin{equation}\label{eq:ch}
    c_h(\xdots):=\min_{z\in\RRd}\sum_{i=1}^{N}\lambda_i h(x_i-z)=\sum_{i=1}^{N}\lambda_i h(x_i-\overline{x}_h(\mathbf{x})).
\end{equation}
Given compact subsets $X_1,\dots,X_N$ of $\RRd$ and $\mu_1,\dots,\mu_N$ probability measures such that $\mu_i\in\Prob(X_i)$, $i=1,\dots, N$, we are interested in the variational problem
\begin{equation}\label{eq:MMhbary}\tag{MM-$h$-bar}
	C_{h{\text{-MM}}}:=\min_{\gamma\in\Pi(\mu_1,\dots,\mu_N)}\int_{\RR^{Nd}} c_h(\xdots)\,d\gamma.
\end{equation}
\begin{remark}
	The $h$-barycenter map $\xh$, and therefore $c_h$, is continuous, hence by standard results on multi-marginal optimal transport, problem \eqref{eq:MMhbary} admits a solution, see Proposition~\ref{prop:existenceMMOT}.
\end{remark}
Given the $h$-barycenter $\xh$, we can lift it to the space of probability measures by introducing the corresponding MMOT problem with cost $c_h$. 
\begin{definition}[$h$-Wasserstein barycenter -- multi marginal formulation]\label{def:hwassbaryMM}
	Let $\overline\gamma$ be a solution of the problem \eqref{eq:MMhbary}. Then its push-forward under the $h$-barycenter map, 
	\begin{equation*}
		\overline{\nu}:=(\overline{x}_h)_{\sharp}\overline{\gamma},
	\end{equation*}
	is called $h$-Wasserstein barycenter of $\mu_1,\dots,\mu_n$ with weights $\lambda_1,\dots,\lambda_N$. 
\end{definition}

In analogy with the definition given by Agueh and Carlier in \cite{AC11} for the case $h=|\cdot|^2$,  we also give the following definition.
\begin{definition}[$h$-Wasserstein barycenter -- coupled two-marginal formulation]\label{def:hwassbary}
	For any probability measures $\mu_1,\dots,\mu_N$ such that $\mu_i\in\Prob(X_i)$, $i=1,\dots,N$, 
	the $h$-Wasserstein barycenter is the optimizer of the variational problem
	\begin{equation}\label{C2Mhbary}\tag{C2M-$h$-bar}
		C_{h\text{-C2M}}:=\inf_{\rho\in\Prob(\RRd)}\sum_{i=1}^{N}W_h(\mu_i, \rho),
	\end{equation}
	where 
	\begin{equation*}
		W_h(\mu_i,\rho):=\min_{\eta\in\Pi(\rho,\mu_i)}\int_{\RR^{2d}}h(x_i-z)d\eta.
	\end{equation*}
\end{definition}

The use of the name \textit{$h$-Wasserstein barycenter} both in Definition \ref{def:hwassbaryMM} and Definition \ref{def:hwassbary} is justified by the following result.
\begin{proposition}\label{prop:equivalence}
	For any $\mu_1,\dots,\mu_N$ probability measures such that $\mu_i\in\Prob(X_i)$, $i=1,\dots,N$, 
	we have that 
	\begin{equation*}
		C_{h\text{-MM}}=C_{h\text{-C2M}}.
	\end{equation*}
	Moreover, $\overline{\nu}$ is a minimizer for the problem \eqref{C2Mhbary} if and only if $\overline\nu=(\overline{x}_h)_\sharp\overline\gamma$, for some minimizer $\overline\gamma$ of the problem \eqref{eq:MMhbary}.
\end{proposition}
The proof of Proposition~\ref{prop:equivalence}, which is analogous to \cite[Proposition 3]{CE10}, is presented in Section~\ref{sec:equivalence}.

\medskip

Under the additional assumptions 
\begin{enumerate}[label=(\subscript{A}{\arabic*}), resume*]
	\item\label{A4}\textit{regularity}: $h\in C^2(\RRd)$,
	\item\label{A5}\textit{non-degeneracy}: $D^2h(z)>0$ for every $z\in \RRd$,
\end{enumerate}
on the function $h$, our main result about the structure of optimal plans for \eqref{eq:MMhbary} is the following:  
\begin{theorem}\label{th:sparsityplan}
Let $\mu_1 \ll \LL^d$. Then there exists a unique optimal plan $\overline{\gamma}$ for the problem \eqref{eq:MMhbary} and for every $i=2,\dots,N$, there exists a measurable map $T_i:X_1\to X_i$ such that
\begin{equation}\label{eq:monge}
    \overline\gamma=(Id,T_2,\dots,T_N)_\sharp\mu_1.
\end{equation}
With the further assumption that all the $\mu_1,\dots,\mu_N \ll \LL^d$, there exist functions $\varphi_i: \mathrm{supp}\mu_i \to\RR$ such that 
\begin{equation}\label{eq:mapT}
    T_i(x_1)=\left(\mathrm{Id}-Dh^{-1}\circ\left(\lambda_i^{-1} D\varphi_i\right)\right)^{-1} \circ \left(\mathrm{Id}-Dh^{-1}\circ\left(\lambda_1^{-1} D\varphi_1\right)\right)(x_1),
\end{equation} for any $i=2,\dots,N$.
\end{theorem}

\begin{remark}
	The \emph{Kantorovich potentials} $\varphi_1,\dots,\varphi_N$ in Theorem~\ref{th:sparsityplan} are $c_h$-conjugate functions, see Definition~\ref{def:conjugate} below.
\end{remark}

\subsection{Proof strategy} 
Consider an optimal $\overline{\gamma}$ for \eqref{eq:MMhbary}. 
Then by Proposition~\ref{prop:optimalplanarecmonotone} its support is a $c_h$-monotone set.
The key ingredient in establishing sparsity of the support of $\overline{\gamma}$ is to show that if $\mu_1 \ll \LL^d$, then the barycenter $\overline{\nu}$ is also absolutely continuous with respect to Lebesgue measure. This is the content of Theorem~\ref{th:absolutecontinuity}, which relies on optimality of $\overline{\gamma}$ in the somewhat weaker form of $c_h$-monotonicity of $\spt\overline{\gamma}$. In fact, we prove a stronger statement: the injectivity of the map $\xh$ from $N$ points in $\spt\overline{\gamma}$ to their barycenter can be quantified in a local inverse Lipschitz estimate, see Proposition~\ref{prop:lowerbounddistancebary} and Corollary~\ref{cor:lowerbounddistancebary}. 
Moreover, the barycenter $\overline{\nu}$ has compact support, being the push-forward of a measure with compact support by a continuous map.

The rest of the proof appeals to classical results in the two-marginal theory, observing that the couplings $\gamma_i:=(\pi_i,\xh)_\sharp\overline\gamma$ are optimal for $W_h(\mu_i,\overline\nu)$, $i=1, \dots, N$, see Corollary~\ref{cor:optimalitytwomarginals}. Hence the celebrated result of Gangbo and McCann \cite{GMC96} ensures that there exist unique maps $f_i$ such that $\mu_i=(f_i)_\sharp\overline\nu$ and $\gamma_i=(f_i,id)_\sharp\overline\nu$ for every $i={1,\dots,N}$. This implies that 
\begin{equation*}
    \overline\gamma=(f_1,\dots,f_N)_\sharp\overline\nu.
\end{equation*}
Now, as $\mu_1 \ll \LL^d$, by the same result of Gangbo--McCann, we know that there exists an optimal map $g_1$ such that $\gamma_1=(id,g_1)_\sharp\mu_1$, and $g_1$ is the a.e.-inverse of $f_1$. Thus
\begin{equation*}
    \overline\gamma=(f_1,\dots,f_N)_\sharp(g_1)_\sharp\mu_1=(id,f_2\circ g_1, \dots, f_N\circ g_1)_\sharp\mu_1.
\end{equation*}
Once the Monge structure of the optimal transport plan is established, the almost-everywhere uniqueness of the maps $T_i = f_i \circ g_1$ then follows from a standard argument: the convex combination of two different optimal plans is still optimal, but cannot be Monge, which contradicts what we just showed.

If we assume in addition that $\mu_1,\dots,\mu_N$ are all absolutely continuous w.r.t.\ $\LL^d$, 
we can appeal to the first-order optimality system for the ($c_h$-conjugate) Kantorovich potentials $\varphi_1, \dots, \varphi_N$, 
\begin{align}\label{eq:optimalitysystem}
	\begin{split}
      \lambda_1 Dh(x_1-\xh(\xvec))&=D\varphi_1(x_1),\\
      &\vdots\\
      \lambda_N Dh(x_N-\xh(\xvec))&=D\varphi_N(x_N),\\
	\end{split}
\end{align}
see Proposition~\ref{prop:firstordoptsystem}.
In particular, if $\xvec\in\spt\overline\gamma$, then 
\begin{equation}\label{eq:baryfunctionx1}
    \xh(\xvec)=x_i-Dh^{-1}\left(\lambda_i^{-1} D\varphi_i(x_i)\right)=:g_i(x_i),
\end{equation}
for every $i=1,\dots,N$, where $g_i$ is the optimal transport map (in the two-marginal sense) from $\mu_i$ to $\overline{\nu}$. With the same argument as before, we know that it is invertible for every $i=1,\dots,N$ (with an a.e.-inverse that we called $f_i$ above). Then, for every $i=2,\dots,N$, 
\begin{align*}
    x_i=g_i^{-1}(\xh(\xvec))=g_i^{-1}(g_1(x_1)),
\end{align*}
which proves \eqref{eq:mapT}. \hfill $\blacksquare$

\medskip
The rest of the article is structured as follows: 
In Section~\ref{sec:preliminaries} we recall basic properties of multi-marginal optimal transport. The equivalence of the coupled two-marginal formulation and MMOT formulation of $h$-barycenters is presented in Section~\ref{sec:equivalence}. In Section~\ref{sec:main} we give the detailed proof of the absolute continuity of the $W_h$-barycenter, which is the main ingredient in the proof of Theorem~\ref{th:sparsityplan}.
\section{Preliminaries}\label{sec:preliminaries}
Let $X_1,\dots, X_N$ such that $X_i:=\RR^{d_i}$ for every $i=1,\dots,N$ and consider $\mu_1,\dots,\mu_N$ such that $\mu_i\in\Prob(X_i)$. \\Let $c:X_1\times\dots\times X_N\to\RR$ be a lower semicontinuous function, which is bounded from below. Then the multi-marginal optimal transport problem associated to the cost $c$ is  
 \begin{equation}\label{eq:MMOT}\tag{MMOT}
     \min_{\gamma\in\Pi(\mu_1,\dots,\mu_N)}\int_{\RR^{Nd}} c(\xdots)\,d\gamma(\xdots),
 \end{equation}
where
 \begin{equation*}
     \Pi(\mu_1,\dots,\mu_N):=\{\gamma\in\Prob(\RR^{Nd}) \, : \, \pi^i_\sharp \gamma = \mu_i \},
 \end{equation*}
 is the set of transport plans.
 
\begin{proposition}\label{prop:existenceMMOT}
The problem \eqref{eq:MMOT} admits a solution.
\end{proposition}
The proof proceeds via standard compactness and lower semicontinuity arguments, see e.g. \cite{F24}.
\begin{definition}\label{def: c-monotonicity}
Let $c:X_1\times\dots\times X_N\to\RR$ be any continuous cost function. Then we say that a set $\Gamma\subset X_1\times\dots\times X_N$ is $c$-monotone if for every $\xvec^1=(x^1_1,\dots,x^1_N)$, $\xvec^2=(x^2_1,\dots,x^2_N)\in\Gamma$ we have that 
\begin{equation*}
    c(x^1_1,\dots,x^1_N)+c(x^2_1,\dots,x^2_N)\le c(x^{\sigma_1(1)}_1,\dots,x^{\sigma_N(1)}_N)+c(x^{\sigma_1(2)}_1,\dots,x^{\sigma_N(2)}_N),
\end{equation*}
where $\sigma_i\in S(2)$, with $S(2)$ the set of permutations of two elements. 
\end{definition}
\begin{remark}
Definition \ref{def: c-monotonicity} is the standard generalization to multi-marginal OT of the known notion of $c$-monotonicity for $2$-marginal OT. We recall that also the notion of   $c$-cyclical monotonicity can be defined  in the  multi-marginal optimal setting (see \cite[Definition 2.2]{KP14}) and that $c$-cyclical monotonicity implies $c$-monotonicity.
\end{remark}
The following is a straightforward consequence of the fact that the support of optimal plans is $c$-cyclically monotone (see Proposition 2.3 in \cite{KP14}).
\begin{proposition}\label{prop:optimalplanarecmonotone}
If $c:X_1\times\dots\times X_N\to\RR$ is continuous and $\gamma$ is optimal for the problem \eqref{eq:MMOT}, then $\spt\gamma$ is $c$-monotone.
\end{proposition}

\begin{proposition}\label{prop:goodinequality}
Let $c:X_1\times\dots\times X_N\to\RR$ be $C^2$ and $\Gamma\subset X_1\times\dots\times X_N$ be a $c$-monotone set. Then, for any set  $S\subseteq\{1,\dots,N\}$, s.t. $S\neq\emptyset,\{1,\dots,N\}$, we have that for any $\yvec=(\ydots)$, $\ytildevec=(\ytildedots)\in\Gamma$
\begin{equation}\label{eq:goodinequality}
    \int_0^1\int_0^1 \sum_{i \in S, j\not\in S}(y_i-\tilde{y}_i)^\intercal D^2_{x_ix_j}c(\yvec(s,t))(y_j-\tilde{y}_j)dtds\le 0,
\end{equation}
where
\begin{equation*}
    y_i(s,t)=\begin{cases} \tilde{y_i}+t(y_i-\tilde{y}_i) \quad \text{if} \ i\not\in S, \\ 
    \tilde{y_i}+s(y_i-\tilde{y}_i) \quad \text{if} \ i\in S. \end{cases}
\end{equation*}

\end{proposition}
\begin{proof}
Let us take $\sigma_1,\dots\sigma_N\in S(2)$, such that 
\begin{equation*}
  \sigma_i(y_i)=
    \begin{cases}
  y_i, \quad \text{if} \ i\not\in S,\\
    \tilde{y}_i, \quad \text{if} \ i\in S,
    \end{cases}
    \ \text{and} \quad \sigma_i(\tilde{y}_i)=
       \begin{cases}
 \tilde{y}_i, \quad \text{if} \ i\not\in S,\\
    y_i, \quad \text{if} \ i\in S,
    \end{cases}
\end{equation*}
and define  $\zvec=(\zdots)=(\sigma_1(y_1),\dots,\sigma_N(y_N))$ and $\tilde{\zvec}=(\tilde{z}_1,\dots,\tilde{z}_N)=(\sigma_1(\tilde{y}_1),\dots,\sigma_N(\tilde{y}_N))$. 
Since $\Gamma$ is $c$-monotone, we have that 
\begin{equation}\label{eq:meanvalue1}
    c(\ydots)-c(\zdots)- (c(\tilde{z}_1,\dots,\tilde{z}_N)-c(\ytildedots))\le 0.
\end{equation}
The result then follows by applying twice the fundamental theorem of calculus. Indeed first by \eqref{eq:meanvalue1} we get 
\begin{equation}\label{eq:meanvalue2}
    \int_0^1 \sum _{i\in S}D_{x_i} c(\yvec(t))(y_i-\tilde{y}_i)dt-\int_0^1 \sum _{i\in S}D_{x_i} c(\ytildevec(t))(y_i-\tilde{y}_i)dt\le 0,
\end{equation}
where 
\begin{align*}
\begin{split}
    y_i(t)=\begin{cases} y_i \quad &\text{if} \ i\not\in S, \\ 
    \tilde{y_i}+t(y_i-\tilde{y}_i) \quad &\text{if} \ i\in S, \end{cases}	
\end{split}
\qquad\text{and}
\begin{split}
	\tilde{y}_i(t)=\begin{cases} \tilde{y}_i \quad &\text{if} \ i\not\in S, \\ 
    \tilde{y_i}+t(y_i-\tilde{y}_i) \quad &\text{if} \ i\in S. \end{cases}
\end{split}
\end{align*}
By applying the fundamental theorem again in \eqref{eq:meanvalue2}, we obtain \eqref{eq:goodinequality}.
\end{proof}
\begin{definition}[$c$-conjugate functions \cite{F24}]\label{def:conjugate}
Let $c:X_1\times\dots\times X_N\to\RR$ be a continuous cost function and $\varphi_i:X_i\to\RR\cup\{-\infty\}$ be a proper function for every $i=1,\dots,N$. 
\begin{enumerate}
    \item For every $i=1,\dots,N$ the $c,i$-transform of $(\varphi_1,\dots,\varphi_{i-1},\varphi_{i+1},\dots,\varphi_N)$ is the function defined by
    \begin{equation*}
        (\varphi_j)_{j\neq i}^{c}(x_i):=\inf_{\substack{\hat{\xvec}_i=(x_1,\dots,x_{i-1},x_{i+1},\dots,x_N)\\\in X_1\times\dots\times X_N}}\left\{ c(\xdots)-\sum_{j\neq i}\varphi_j(x_j)  \right\};
    \end{equation*}
    \item The functions $\varphi_1,\dots,\varphi_N$ are called $c$-conjugate if 
    \begin{equation*}
        \varphi_i=(\varphi_j)_{j\neq i}^{c}, \quad \text{for every} \ i=1,\dots,N.
    \end{equation*}
\end{enumerate}
\end{definition}
It is a well known fact in OT that there exists a dual formulation of the problem, see \cite{RR98}. The version of the duality theorem we provide can be found in \cite{F24}. 
\begin{theorem}\label{prop:duality}
The problem \eqref{eq:MMOT} admits the following dual formulation
\begin{equation*}
    \eqref{eq:MMOT}=\sup_{\varphi_1,\dots,\varphi_N\in\mathcal{A}(c)}\sum_{i=1}^{N}\int_{\RRd}\varphi_id\mu_i,
\end{equation*}
where 
\[\mathcal{A}(c):=\{ (\varphi_1,\dots,\varphi_N)\in C_{0}(X_1)\times\cdots\times C_{0}(X_N) \, : \, \varphi_1\oplus\cdots\oplus\varphi_N\le c\}.  \]
If $X_1,\dots,X_N$ are compact subsets of $\RR^{d_1},\dots,\RR^{d_N}$, there exists a maximizer $\Phi=(\varphi_1,\dots,\varphi_N)\in\mathcal{A}(c)$ and the optimal $\varphi_1,\dots,\varphi_N$ are called Kantorvich potentials. Moreover we can choose the optimal $\Phi=(\varphi_1,\dots,\varphi_N)\in\mathcal{A}(c)$, such that all the $\varphi_i$s are $c$-conjugate functions.
\end{theorem}

\begin{corollary}\label{cor:firstorderopt}
Let $\gamma\in\Pi(\mu_1,\dots,\mu_N)$ be optimal for \eqref{eq:MMOT} and $\varphi_1,\dots,\varphi_N$ be Kantorovich potentials. Then for every point $\mathbf{x}=(\xdots)\in\spt\gamma$ where $c$ is partially differentiable at $\xvec$ with respect to $x_i$ and $\varphi_i$ is differentiable at $x_i$, we have 
\[D_{x_i}c(\xdots)= D\varphi_i(x_i).\]
\end{corollary}
\begin{proof}
This is also a standard result (see for instance \cite{F24}) due to the fact that, if $\varphi_1,\dots,\varphi_N$ are the Kantorovich potentials, the nonnegative function $c-\varphi_1\oplus\cdots\oplus\varphi_N$ is equal $0$ on $\spt \gamma$ and therefore any $\xvec\in\spt\gamma$ is a minimum point of such a function.
\end{proof}
\section{Equivalence with coupled two-marginal formulation}\label{sec:equivalence}
\begin{proof}[Proof of Proposition~\ref{prop:equivalence}]
 We first show that $C_{h\text{-MM}}\ge C_{h\text{-C2M}}$. Let $\overline\gamma$ be optimal for \eqref{eq:MMhbary} and let $\overline\nu=(\overline{x}_h)_\sharp\overline\gamma$. For any $i$ we define $\gamma_i:=(\pi_i,\xh)_\sharp\overline\gamma$, where $\pi_i(\xdots)=x_i$, then $\gamma_i\in\Pi(\mu_i,\overline\nu)$ and
 \begin{align}
    \notag \int_{\RR^{Nd}}\lambda_ih(x_i-\overline{x}_h(\mathbf{x}))d\overline\gamma(\xvec)&=\int_{\RR^{Nd}}\lambda_ih\left(\pi_i(\mathbf{x})-\overline{x}_h(\mathbf{x})\right)d\overline\gamma(\xvec)\\
    \notag &=\int_{\RR^{2d}}\lambda_ih(x_i-z)d\gamma_i(x_i,z)\\
     \label{eq:MMgraterC2M}&\ge \lambda_iW_h(\overline\nu,\mu_i).
 \end{align}
 By summing over $i$ we obtain the desired inequality.\\
 Let us prove the converse inequality $C_{h\text{-MM}}\le C_{h\text{-C2M}}$. Let $\rho\in\Prob(\RRd)$ and let $\gamma_i$ be optimal for $\lambda_iW_h(\rho,\mu_i)$ and $\{\gamma_i^{(z)}\}_{z\in\RRd}$ be the disintegration of $\gamma_i$ with respect to the first marginal, i.e. $\gamma_i^{(z)}\in\Prob(\RRd)$ for every $z\in\RRd$ and 
\[d\gamma_i(x_i,z)=d\gamma_i^{(z)}(x_i)d\rho(z).\]
If $\gamma$ is such that  
\begin{equation}\label{eq:gammadisint}
    d\gamma(\xdots)=\int_{\RRd}d\gamma_1^{(z)}(x_1)\cdots d\gamma_N^{(z)}(x_N)d\rho(z),
\end{equation}
then
\begin{align}
\notag\sum_i \lambda_i W_h(\mu_i,\rho)&=\sum_i\lambda_i\int_{\RR^{2d}}h(x_i-z) d\gamma_i\\\notag &=\sum_i\lambda_i\int_{\RR^{2d}}h(x_i-z)d\gamma_i^{(z)}(x_i)d\rho(z)\\\notag&=\int_{\RR^{d(N+1)}}\sum_i\lambda_ih(x_i-z)d\gamma_1^{(z)}(x_1)\cdots d\gamma_N^{(z)}(x_N)d\rho(z)\\\label{eq:minimalitybary}&\ge \int_{\RR^{Nd}}\sum_i\lambda_ih(x_i-\overline{x}_h(\mathbf{x})) \int_{\RRd}d\gamma_1^{(z)}(x_1)\cdots d\gamma_N^{(z)}(x_N)d\rho(z)\\\label{eq:last}&\ge C_{h\text{-MM}}.
\end{align}
The claim follows by the arbitrariness of $\rho$. \\
We conclude by proving the equivalence of the minimizers. Let $\overline\gamma$ be optimal for \eqref{eq:MMhbary}. Being inequality in \eqref{eq:MMgraterC2M} in fact an equality, $\overline\nu=(\overline{x}_h)_\sharp \gamma$ has to be optimal for \eqref{C2Mhbary}. \\
Vice versa, let $\overline\nu$ be optimal for the problem \eqref{C2Mhbary} and define $\hat\gamma\in\Pi(\mu_1,\dots,\mu_N,\overline\nu)$ such that
\begin{equation*}
d\hat\gamma(\xdots,z):=d\gamma_1^{(z)}(x_i)\cdots d\gamma_N^{(z)}d\overline\nu(z).
\end{equation*}
Clearly, by \eqref{eq:gammadisint}, we have that $\gamma=(\pi_1,\dots, \pi_N)_{\sharp}\hat\gamma$. Since the inequalities \eqref{eq:minimalitybary} and \eqref{eq:last} are in fact equalities, $\gamma$ has to be optimal. 
But the inequality \eqref{eq:minimalitybary} is an equality if and only if $z=\overline{x}_h(\mathbf{x})$ on $\spt\gamma$, then we also have that $\hat\gamma=(\overline{x}_h,id)_\sharp\gamma$ and thus $\overline\nu=(\pi_0)_\sharp\hat\gamma=(\overline{x}_h)_\sharp\gamma$.
\end{proof}
\begin{remark}
Proposition \ref{prop:equivalence} ensures the existence of a minimizer for the problem \eqref{C2Mhbary}. Without appealing to the formulation as multi-marginal optimal transport problem, an existence result for problems of the form \eqref{C2Mhbary} can be found in \cite{CE10}.
\end{remark}
\begin{corollary}
\label{cor:optimalitytwomarginals}
Let $\overline\gamma$ be optimal for the problem \eqref{eq:MMhbary}. Then  $\gamma_i:=(\pi_i,\xh)_\sharp\overline\gamma\in \Pi(\mu_i,\overline\nu)$ and is optimal for $W_h(\mu_i,\overline\nu)$, where $\overline\nu:=(\xh)_\sharp\overline\gamma$ is the $h$-Wasserstein barycenter.
\end{corollary}
\begin{proof}
The fact that $\gamma_i:=(\pi_i,\xh)_\sharp\overline\gamma\in \Pi(\mu_i,\overline\nu)$ is straightforward from the definition. Moreover
\begin{align*}
\sum_i \lambda_i W_h(\mu_i,\overline\nu)&\le\sum_i\int_{\RR^{2d}}\lambda_ih(x_i-z) d\gamma_i(x_i,z)\\&=\sum_i\int_{\RR^{Nd}}\lambda_ih(\pi_i(\xvec)-\xh(\xvec)) d\overline\gamma(\xvec)\\&=\sum_i \lambda_i W_h(\mu_i,\overline\nu),
\end{align*}
where the last equality is a consequence of Proposition  \ref{prop:equivalence}.
Since $\lambda_i W_h(\mu_i,\overline\nu)\le \lambda_i\int_{\RR^{2d}}h(x_i-z) d\gamma_i$ for every $i$, the equality above yields $\lambda_i W_h(\mu_i,\overline\nu)= \lambda_i\int_{\RR^{2d}}h(x_i-z) d\gamma_i$ and thus the optimality of $\gamma_i$.
\end{proof}
\section{Absolute continuity of the barycenter and proof of the main result}\label{sec:main}
Lemma \ref{lem:firstinequality} below already contains a central idea of our proof that optimal plans are of Monge form. As explained in the introduction, we seek to show that the following highly non-injective map, defined by \eqref{eq:hbary}, from $\RR^{Nd}$ to $\RR^d$,
\begin{equation} \label{eq:barymap}
    \xvec=(\xdots)\mapsto \xh(\xvec),
\end{equation}
is injective on the support of optimal plans, and has a locally Lipschitz inverse. This implies that the local dimension of the support is just $d$, not $N\!\cdot\! d$.
We will prove this by a monotonicity argument, which is easy to understand for the quadratic cost $h(x_i-z)=\tfrac12|x_i-z|^2$. In this case, it simply says that the barycenter map \eqref{eq:barymap} is strictly monotone in each $y_i$ with a quantitative lower bound, 
\begin{equation} \label{eq:firstineq2}
       (y_{i_0}-\tilde{y}_{i_0})^T\left(\xh(\yvec)-\xh(\ytildevec)\right)\ge \Lambda |y_{i_0}-\tilde y_{i_0}|^2
\end{equation}
for some $\Lambda>0$. We claim that this inequality directly follows from
\begin{enumerate}[label=(\roman*)]
    \item $c$-monotonicity of the support of optimal plans,
    \item the inequality \eqref{eq:goodinequality} for any two points in a $c$-monotone set.
\end{enumerate}
Indeed, in the quadratic case we have $\xh(\xvec)=\sum_{j=1}^N \lambda_j x_j$ and the mixed ($i\neq j$) second derivative matrix $D^2_{x_ix_j}c(\xvec)$ equals $-\lambda_i\lambda_j \mathrm{Id}$. In particular, it is constant, so that \eqref{eq:goodinequality} reduces to 
\begin{equation}\label{eq:goodinequality2}
     \sum_{i\in S,j\not\in S} (y_i-\ytilde_i)^T D^2_{x_ix_j} c(\xvec) (y_j-\ytilde_j) \le 0
\end{equation}
where $x\in\RR^{dN}$ is arbitrary. Consequently 
\begin{align*}
    (y_{i_0}-\tilde{y}_{i_0})^T\left(\xh(\yvec)-\xh(\ytildevec)\right) &= (y_{i_0}-\tilde{y}_{i_0})^T \sum_{j=1}^N \lambda_j(y_j-\ytildevec_j) \\
    &= |y_{i_0}-\tilde{y}_{i_0}|^2 \underbrace{- \frac{1}{\lambda_{i_0}} \sum_{ j\neq i_0} (y_{i_0}-\tilde{y}_{i_0})
    D^2_{x_{i_0}x_j}c(x) (y_j-\ytilde_j)}_{\ge 0, \text{ by } \eqref{eq:goodinequality2}}, 
\end{align*}
establishing \eqref{eq:firstineq2}. Our first lemma below locally generalizes this estimate to $C^2$ costs $h$ with locally positive Hessian.

Throughout this section we assume that $h:\RRd\to\RR$ satisfies assumptions \ref{A1}-\ref{A3} and, in addition,
\begin{enumerate}[label={(\subscript{A}{4}})]
    \item\textit{Regularity}: $h\in C^2(\RRd)$.
\end{enumerate}

\begin{remark}\label{rmk:diffbary}
	Under assumptions \ref{A1}-\ref{A4} on $h$, due to its optimality, the point $\xh(\xdots)$ is the only solution of 
	\begin{equation}\label{eq:firstorderoptimalitybary}
		\sum_{i=1}^{N}\lambda_i Dh(x_i-z)=0,
	\end{equation}
	for any given $\xvec=(\xdots)\in\RR^{Nd}$. 
	Let us define the function $F:\RR^{(N+1)d}\to\RRd$ by 
	\begin{equation*}
		F(\mathbf{x},z):=\sum_{i=1}^{N}\lambda_i Dh(x_i-z).
	\end{equation*}
	If $\xvec=(\xdots)\in\RR^{Nd}$ is such that $\det(D^2h(x_i-\xh(\xvec)))\neq0$ (and therefore $D^2h(x_i-\xh(\xvec))$ is a strictly positive matrix) for some $i\in\{1,\dots,N\}$, then $D_zF(\xvec,\xh(\xvec))$ is invertible and the Implicit Function Theorem implies that there exists an open neighborhood $U_{\xvec}$ of $\xvec$, such that $\xh$ satisfies \eqref{eq:firstorderoptimalitybary} on $U_{\xvec}$ and $\xh\in C^1(U_{\xvec})$. For any $i=1,\dots,N$, we then have that
	\begin{align}\label{eq:diffbary}
		\notag 
		D_{x_i} \xh (\ydots)=&-D_{z} F(\yvec, \xh(\mathbf{y}))^{-1}D_{x_i} F(\mathbf{y}, \xh(\mathbf{y}))\\
		&= \Bigl(\sum_{k=1}^N \lambda_kD^2h(y_k-\xh(\yvec))\Bigr)^{-1}\lambda_iD^2h(y_i-\xh(\yvec)),
	\end{align}
	for any $\yvec=(\ydots)\in U_{\xvec}$. 
	Moreover the function $c_h:\RR^{dN}\to\RR$, defined in \eqref{eq:ch}, belongs to $C^2(U_{\xvec})$. Indeed, for every $\yvec\in U_{\xvec}$, differentiating we get
	\begin{align}\label{eq:firstderivativecost}
		\notag D_{x_i}c_h(\ydots)&= \lambda_iDh(y_i-\xh(\yvec))-\left(\sum_{k=1}^{N}\lambda_kDh(y_k-\xh(\yvec))\right)D_{x_i}\xh(\mathbf{y})\\&= \lambda_iDh(y_i-\xh(\yvec)),
	\end{align}
	since $\sum_{k=1}^{N}\lambda_kDh(y_k-\xh(\yvec))=0$, by \eqref{eq:firstorderoptimalitybary}.
	Differentiating again, if $i\not = j$, one has
	\begin{align}\label{eq:mixedsecondderivative}
		\notag D^2_{x_ix_j}c_h(\ydots)&= -\lambda_iD^2h(y_i-\xh(\yvec))D_{y_j} \xh (\ydots)
		\\ &= -\lambda_i\lambda_jD^2h(y_i-\xh(\xvec))\left(\sum_{k=1}^N \lambda_kD^2h(y_k-\xh(\yvec))\right)^{-1}D^2h(y_j-\xh(\xvec))
	\end{align}
	and, if $i=j$
	\begin{align*}\label{eq:secondderivativecost}
		D^2_{x_i}c_h(\ydots)
		&=\lambda_iD^2h(y_i-\xh(\yvec)) \left(\1-\left(\sum_{k=1}^N \lambda_kD^2h(y_k-\xh(\yvec)) \right)^{-1}\lambda_iD^2h(y_i-\xh(\yvec)) \right).
	\end{align*}
\end{remark}

\begin{description}
 \item[Notation] In order to lighten the notation we fix 
 \begin{align*}
     &H(\xvec):=\sum_{k=1}^N \lambda_kD^2h(x_k-\xh(\xvec)),\\
     &M_i(\xvec):=D^2h(x_i-\xh(\xvec)), \quad \text{for any} \ i\in\{1,\dots,N\}.
 \end{align*}
\end{description}
 \begin{lemma}\label{lem:firstinequality}
 Let $\overline\gamma$ be $c_h$-monotone (see Defintion \ref{def: c-monotonicity}) and let $\xvec=(\xdots)\in\RR^{Nd}$, such that $\det D^2h(x_{i_0},\xh(\xvec))\neq 0$ (and thus $D^2h(x_{i_0}-\xh(\xvec))>0$ ), for some $i_0\in\{1,\dots,N\}$. Then for every $\eps>0$, there exists $r>0$ such that  for every $\yvec, \ytildevec\in\spt\overline\gamma\cap B(\xvec,r)$
\begin{equation}\label{eq:firstinequality}
    (y_{i_0}-\tilde{y}_{i_0})^T\left(D^2h(x_{i_0}-\xh(\xvec))\right) \left(\xh(\yvec)-\xh(\ytildevec)\right)\ge \Lambda_{i_0} |y_{i_0}-\tilde y_{i_0}|^2-\eps N(1+|M_{i_0}(\xvec)|)|\yvec-\ytildevec|^2,
\end{equation}
where $\Lambda_{i_0}>0$ is the minimum eigenvalue of the matrix  $D^2h(x_{i_0}-\xh(\xvec)H(\xvec)^{-1}D^2h(x_{i_0}-\xh(\xvec))$. 
 \end{lemma}
In the quadratic case, the matrix $D^2h$ is the identity and the left hand side of \eqref{eq:firstinequality} reduces to that of \eqref{eq:firstineq2}. Moreover as will become clear in the proof, the additional term on the right as compared to \eqref{eq:firstineq2} comes from the fact that in the non-quadratic case the matrix $D^2_{x_ix_j}c_h$ is no longer constant. 

\begin{proof}
 For the sake of simplicity, we assume that $i_0=1$. We call   $A_{1j}:=D^2_{x_1x_j}c_h(\xdots)$ and $B_j:= D_{x_j}\xh(\xdots)$.
 By Remark \ref{rmk:diffbary} we know that $A_{1j}, B_j$ are well defined and that for every $\eps>0$, there exists $r>0$ such that
 \begin{align}
     \label{eq:continuitysecondderivativec}&|D^2_{x_1x_j}c_h(\yvec)-A_{1j}|<\eps, \quad \text{and} \\
     \label{eq:continuityfirstderivativebary}&|D_{x_j}\xh(\yvec)-B_j|<\eps,
 \end{align}
for every $j\in\{1,\dots,N\}$ and for every $\yvec=(\ydots)\in B(\xvec,r)$.\\
So we fix $\eps>0$ and we take $\yvec=(\ydots),\ytildevec=(\ytildedots)\in \spt\overline\gamma\cap B(\xvec,r)$. Since $\overline\gamma$ is $c_h$-monotone, thanks to Proposition~\ref{prop:goodinequality}, by considering $S=\{1\}$ in \eqref{eq:goodinequality},  we get 
\begin{equation*}
    \int_0^1\int_0^1 \sum_{j=2}^N(y_1-\tilde{y}_1)^T D^2_{x_1x_j}c_h(\yvec(s,t))(y_j-\tilde{y}_j)dtds\le 0,
\end{equation*}
which implies, thanks to \eqref{eq:continuitysecondderivativec},
\begin{equation}\label{eq:goodinequalityintheproof}
    \sum_{j=2}^N(y_1-\tilde{y}_1)^T A_{1j}(y_j-\tilde{y}_j)\le\eps |y_1-\tilde{y}_1|\sum_{j=2}^N|y_j-\tilde{y}_j|.
\end{equation}
Let $\yvec(t)$ defined such that $y_j(t)=\tilde{y}_j+t(y_j-\tilde y_j)$, for every $j$, for every $t\in[0,1]$. If we multiply both side of the equality
\begin{equation*}
    \xh(\yvec)-\xh(\ytildevec)=\int_0^1\sum_{j=1}^ND_{x_j}\xh(\yvec(t))(y_j-\tilde y_j)dt,
\end{equation*}
by $(y_1-\tilde{y}_1)^T D^2h(x_1-\xh(\xvec))$, 
we get 
\begin{align}
  \notag&(y_1-\tilde{y}_1)^T(D^2h(x_1-\xh(\xvec)) \left(\xh(\yvec)-\xh(\ytildevec)\right)\\ 
\label{eq:goodinequalityintheproof2} &= \int_0^1\sum_{j=1}^N(y_1-\tilde{y}_1)^T (D^2h(x_1-\xh(\xvec))D_{x_j}\xh(\yvec(t))(y_j-\tilde y_j)dt\\
   \notag &\overset{\eqref{eq:continuityfirstderivativebary}}{\ge}\sum_{j=1}^N(y_1-\tilde{y}_1)^T \left(D^2h(x_1-\xh(\xvec))
 \right)B_j(y_j-\tilde y_j)-\eps|y_1-\tilde{y}_1||D^2h(x_1,\xh(\xvec))|\sum_{j=2}^N|y_j-\tilde{y}_j|.
\end{align}
We observe that, when $j=1$, we have
\begin{equation*}
   D^2h(x_1-\xh(\xvec))B_1 \stackrel{\eqref{eq:diffbary}}{=}D^2h(x_1-\xh(\xvec))H(\xvec)^{-1}D^2h(x_1-\xh(\xvec))>0,
\end{equation*}
since $H(\xvec)>0$, being the sum of a strictly positive definite matrix ($\lambda_1D^2h(x_1-\xh(\xvec))$) and some nonnegative definite matrices ($\lambda_kD^2h(x_k-\xh(\xvec))$, $k\ge 2$) . Then 
\begin{equation*}
  (y_1-\tilde{y}_1)^T D^2h(x_1-\xh(\xvec))B_1(y_1-\tilde y_1)\ge \Lambda_1 |y_1-\tilde{y}_1|^2,
\end{equation*}
where $\Lambda_1>0$ is the smallest eigenvalue of $M_1^T H^{-1} M_1$.\\
Now we notice that, when $j\ge2$, 
\begin{equation*}
D^2h(x_1-\xh(\xvec))B_j\stackrel{\eqref{eq:mixedsecondderivative}}{=}D^2_{x_1x_j}c(\xdots)=-A_{1j},
\end{equation*}
and then
\begin{align*}
   \sum_{j=2}^N(y_1-\tilde{y}_1)^T D^2h(x_1-\xh(\xvec))B_j(y_j-\tilde y_j)&=-\sum_{j=2}^N(y_1-\tilde{y}_1)^T A_{1j}(y_j-\tilde y_j)\\&\overset{\eqref{eq:goodinequalityintheproof}}{\ge} -\eps |y_1-\tilde{y}_1|\sum_{j=2}^N|y_j-\tilde{y}_j|\\&\ge - \eps N |\yvec-\ytildevec|^2.
\end{align*}
Finally 
\begin{equation*}
    -\eps|y_1-\tilde{y}_1||D^2h(x_1-\xh(\xvec))|\sum_{j=2}^N|y_j-\tilde{y}_j|\ge -\eps N |M_1(\xvec)||\yvec-\ytildevec|^2.
\end{equation*}
Putting all together in \eqref{eq:goodinequalityintheproof2} we obtain
\begin{equation*}
    (y_1-\tilde{y}_1)^T\left(D^2h(x_i-\xh(\xvec))\right) \left(\xh(\yvec)-\xh(\ytildevec)\right)\ge \Lambda_1 |y_1-\tilde y_1|^2-\eps N(1+M_1(\xvec))|\yvec-\ytildevec|^2.
\end{equation*}
 \end{proof}
 \begin{proposition}\label{prop:lowerbounddistancebary}
 Let $\xvec=(\xdots)\in\RR^{Nd}$ such that $D^2h(x_1-\xh(\xvec))>0$, for every $i\in\{1,\dots,N\}$. Let $\overline\gamma$ be optimal for the problem \eqref{eq:MMhbary}. Then, there exist $r(\xvec), L(\xvec)>0$ such that  
 \begin{equation}\label{eq:inverselipschitzlocal}
     |\xh(\yvec)-\xh(\ytildevec)|\ge L(\xvec) |\yvec-\ytildevec|,
 \end{equation}
 for every $\yvec=(\ydots),\ytildevec=(\ytildedots)\in \spt\gamma\cap B(\xvec,r(\xvec))$.
 \end{proposition}
 \begin{proof}
 
By applying Lemma \ref{lem:firstinequality} for every $i=1,\dots,N$ and summing \eqref{eq:firstinequality} over $i$, one gets that for every $\eps>0$ there exists $r>0$ such that
\begin{equation*}
    \sum_{i=1}^N (y_i-\tilde{y}_i)^T D^2h(x_i-\xh(\xvec)) \left(\xh(\yvec)-\xh(\ytildevec)\right)\ge \sum_{i=1}^N \Lambda_i |y_i-\tilde y_1|^2-\eps N\sum_{i=1}^N(1+|M_i(\xvec)|)|\yvec-\ytildevec |^2,
\end{equation*}
for every $\yvec,\ytildevec\in B(\xvec,r)$, and therefore 
\begin{equation*}
NM(\xvec)|\yvec-\ytildevec||\xh(\yvec)-\xh(\ytildevec)|\ge \left(\Lambda_m-\eps N^2(1+M(\xvec))\right)|\yvec-\ytildevec|^2,
\end{equation*}
where $M(\xvec):=\max_i \{|M_i(\xvec)|\}$ and $\Lambda_m:=\min_i \{\Lambda_i\}$.
Then, by choosing $\eps>0$ such that $\Lambda_m-\eps N^2(1+M(\xvec))>0$ and a suitable $r(\xvec)>0$, we get \eqref{eq:inverselipschitzlocal}.
 \end{proof}
\begin{description}\item[Assumptions]
  From now on, we also have to assume that $h:\RRd\to\RR$ satisfies
 \begin{enumerate}[label={(\subscript{A}{5}})]
 \item $D^2h(z)>0$ for every $z\in \RRd$.
 \end{enumerate}
in addition to assumptions \ref{A1}-\ref{A4}.
\end{description}

\begin{corollary}\label{cor:lowerbounddistancebary}
Let $\overline\gamma$ be $c_h$-monotone. Then there exists a constant $L>0$ such that for every $\yvec=(\ydots),\ytildevec=(\ytildedots)\in \spt\overline\gamma$ such that 
 \begin{equation}\label{eq:inverselipschitz}
     |\xh(\yvec)-\xh(\ytildevec)|\ge L |\yvec-\ytildevec|. 
 \end{equation}
 \end{corollary}
 \begin{proof}
 We observe that  $\spt\overline\gamma\subset \bigcup_{\xvec\in\spt\overline\gamma}B(\xvec, r(\xvec))$, where $B(\xvec, r(\xvec))$ are the ones of Proposition \ref{prop:lowerbounddistancebary}. By compactness there exist $\xvec^1,\dots,\xvec^K$, for some $K\in\NN$, such that of $\spt\overline\gamma\subset \bigcup_{k=1}^KB(\xvec^k, r(\xvec^k))$. Therefore it is enough to choose $L=\min_{k}\{L(\xvec^k)\}$.
 \end{proof}
 \begin{lemma}\label{lemma:diameter}
 Let $\overline\gamma$ be $c_h$-monotone. If $E\subset\RR^d$ is such that $\diam(E)<\delta$ for some $\delta>0$, then 
 \[\diam(\pi_1(\xh^{-1}(E)\cap\spt\overline\gamma) )<\frac \delta L,\]
 where $L>0$ is the same of \eqref{eq:inverselipschitz}. 
 \end{lemma}
 \begin{proof}
Let us call $\tilde E = \xh^{-1}(E)\cap\spt\overline\gamma$. If $\tilde E=\emptyset$, we set $\diam(\pi_1(\tilde E))=0$, which is consistent with the definition of Hausdorff measure. Otherwise let $w,\tilde w\in\tilde E$. Then there exist $z,\tilde z \in E$ such that $w\in \xh^{-1}(\{z\})\cap\spt\overline\gamma$ and $\tilde w\in \xh^{-1}(\{\tilde z\})\cap\spt\overline\gamma$. By Corollary \ref{cor:lowerbounddistancebary} we have that 
 \[|\pi_1(w)-\pi_1(\tilde w)|\le|w-\tilde w|\le \frac{1}{L} |z-\tilde z|\le \frac\delta L .\]
 \end{proof}
 \begin{remark}
 We notice that Corollary \ref{cor:lowerbounddistancebary} implies that given an optimal plan $\overline\gamma$, the set $\xh^{-1}(\{z\})\cap\spt\overline\gamma$ is a singleton (if not empty)  for every $z\in\RRd$, that is, the function $\xh$ restricted to $\spt\overline\gamma$ is invertible and the inverse is Lipschitz. 
 \end{remark}
 \begin{theorem}\label{th:absolutecontinuity}
Let $\overline\gamma$ be $c_h$-monotone and  $\mu_1\ll \LL^d$. Then 
 \begin{equation*}
     \overline\nu:=(\xh)_{\sharp}\gamma\ll \LL^d.
 \end{equation*}
 Notice that by Proposition \ref{prop:equivalence}, $\overline\nu$ is an $h$-Wasserstein barycenter (see Definition \ref{def:hwassbary}), i.e.  a solution of \eqref{C2Mhbary}.
 \end{theorem}
 \begin{proof}
 Let us consider a set $E$ such that $\LL^d(E)=0$. Then we have
 \begin{align*}
(\xh)_\sharp\overline\gamma(E)=\overline\gamma(\xh^{-1}(E)\cap\spt\overline\gamma)&\le \overline\gamma(\pi_1(\xh^{-1}(E)\cap\spt\overline\gamma)\times\RR^{(N-1)d})\\&=\mu_1(\pi_1(\xh^{-1}(E)\cap\spt\overline\gamma)),
 \end{align*}
 We want to prove that $\mu_1(\pi_1(\xh^{-1}(E\cap\xh(\spt\overline\gamma)))=0$. Since  $\mu_1\ll \LL^d$, it is enough to prove that \begin{equation}\label{eq:hausdorffzero}
     \LL^d(\pi_1(\xh^{-1}(E)\cap\spt\overline\gamma))=\HH^d(\pi_1(\xh^{-1}(E)\cap\spt\overline\gamma))=0.
 \end{equation}
 Let $\{E_n\}_{n}$ be a countable cover of the set $E$, such that $\diam(E_n)<\delta$ for every $n$, for some $\delta>0$. Then $\{\pi_1(\xh^{-1}(E_n)\cap\spt\overline\gamma)\}_n$ is a cover of $\pi_1(\xh^{-1}(E)\cap\spt\overline\gamma)$ such that $\diam(\pi_1(\xh^{-1}(E_n)\cap\spt\overline\gamma))<\frac \delta L$, with $L>0$, by Lemma \ref{lemma:diameter}. By definition of Hausdorff measure this implies \eqref{eq:hausdorffzero}. For completeness we include the details: note that 
\begin{align*}
    \HH^d_{\frac \delta L}(\pi_1(\xh^{-1}(E)\cap\spt\overline\gamma))&\le c_d\sum\diam (\pi_1(\xh^{-1}(E_n)\cap\spt\overline\gamma))^d\\&\le c_d \frac{1}{L^d}\sum\diam (E_n)^d.
\end{align*}
By taking the infimum over all the countable covers $\{E_n\}_{n}$  with diameter less than $\delta$, \[\HH^d_{\frac \delta L}(\pi_1(\xh^{-1}(E)\cap\spt\overline\gamma))\le \frac{1}{L^d} \HH^d_{\delta}(E),\] and, passing limit for $\delta\to 0$, $\HH^d(\pi_1(\xh^{-1}(E)\cap\spt\overline\gamma))\le \frac{1}{L^d}\HH^d(E)=0$.
 \end{proof}

\begin{proposition}\label{prop:firstordoptsystem}
Let us assume that $\mu_1,\dots,\mu_N \ll \LL^d$.  Then, given an optimal plan $\overline\gamma$ for \eqref{eq:MMhbary}, the following first order optimality system holds for $\overline\gamma$-a.e. $\xvec$:
 \begin{equation}\label{eq:firstordoptsystem}
     \begin{cases}
      \lambda_1Dh(x_1-\xh(\xvec))=D\varphi_1(x_1),\\
      \vdots\\
      \lambda_NDh(x_N-\xh(\xvec))=D\varphi_N(x_N),\\
     \end{cases}
 \end{equation}
 where $\varphi_1,\dots,\varphi_N$ are $c_h$-conjugate Kantorovich potentials.
\end{proposition}

\begin{proof}
We first observe that Proposition \ref{prop:duality} applies to  the problem \eqref{eq:MMhbary} with cost $c_h$. Therefore we can choose the Kantorovich potentials $\varphi_i$ to be $c_h$-conjugates (see Definition \ref{def:conjugate}). As $h$ satisfies assumption \ref{A5}, $c_h$ is continuously differentiable everywhere (see Remark \ref{rmk:diffbary}), and thus locally Lipschitz (more precisely Lipschitz on the compact set $\spt\overline\gamma$). The $\varphi_i$'s are therefore also locally Lispchitz, being the infimum of a family of locally Lipschitz functions with the same Lipschitz constant. This implies that every $\varphi_i$ is differentiable for $\LL^d$-a.e. $x_i$, and thus $\mu_i$-a.e. by absolute continuity. Since $\overline\gamma\in\Pi(\mu_1,\dots,\mu_N)$, Corollary \ref{cor:firstorderopt} can be applied for $\overline\gamma$-a.e. $\xvec$. The system of equations \eqref{eq:firstordoptsystem} is obtained noticing that, as showed in \eqref{eq:firstderivativecost}, $D_{x_i}c(\xdots)=\lambda_iDh(x_i-\xh(\xvec))$ for every $i$.
\end{proof}
{
\addtocontents{toc}{\setcounter{tocdepth}{-10}}
\section*{Acknowledgments}
{\small CB was funded by \emph{Deutsche Forschungsgemeinschaft (DFG -- German Research Foundation) -- Project-ID 195170736 -- TRR109}.  

This work was partially carried out while TR was a visiting research fellow at TU München, funded by 
\emph{Deutsche Forschungsgemeinschaft (DFG -- German Research Foundation) -- Project-ID 195170736 -- TRR109}.
TR would like to thank the Department of Mathematics at TUM for the kind hospitality. 
}
\addtocontents{toc}{\setcounter{tocdepth}{1}}
}

\bigskip
\end{document}